\documentclass[11pt]{article}
\usepackage{geometry}                
\usepackage[dvipdfmx]{graphicx}
\usepackage{amsmath,amssymb,amsthm}
\makeatletter
\@addtoreset{equation}{section}

\makeatother
\newtheorem{theo}{Theorem}
\newtheorem{defi}{Definition}
\newtheorem{lemm}{Lemma}
\newtheorem{coro}{Corollary} 

\title{Solovay reduction and continuity}
\author{Masahiro Kumabe${}^{1}$, Kenshi Miyabe${}^{2}$, Yuki Mizusawa${}^{3}$\footnote{The corresponding author.}, Toshio Suzuki${}^{4}$\footnote{This work was supported by JSPS KAKENHI
Grant Number JP16K05255.}
\\ 
{\small 1: Faculty of Liberal Arts, The Open University of Japan, Japan}
\\
kumabe@ouj.ac.jp
\\
{\small 2: School of Science and Technology, Meiji University, Japan}
\\
miyabe@meiji.ac.jp
\\
{\small 3,4: Department of Mathematical Sciences, Tokyo Metropolitan University, Japan}
\\
houji6@gmail.com, toshio-suzuki@tmu.ac.jp
}

\begin{document}
\maketitle

\begin{abstract}
The objective of this study is a better understanding of the relationships between 
reduction and continuity. 
Solovay reduction is a variation of Turing reduction based on the distance of two real numbers. 
We characterize Solovay reduction by the existence of a certain real function that is computable (in the sense of computable analysis) and Lipschitz continuous. 
We ask whether there exists a reducibility concept that corresponds to H\"{o}lder continuity. The answer is affirmative. 
We introduce quasi Solovay reduction and characterize this new reduction via H\"{o}lder continuity. In addition, we separate it from Solovay reduction and Turing reduction and investigate the relationships between complete sets and partial randomness. 
\end{abstract}



\section{Introduction}

We would like to get a better understanding of the relationships between 
reduction and continuity. 

Suppose that $\alpha$ and $\beta$ are left-c.e. real numbers. A precise definition is as follows. 

\begin{defi}
Suppose that $\alpha$ is a real number. We say $\alpha$ is \emph{left-c.e.} if the left cut of $\alpha$, $L(\alpha) := \{q \in \mathbb{Q} | q < \alpha\}$, is computably enumerable. 
\end{defi}

Observe that there exists a function $f :(-\infty, \beta) \to (-\infty, \alpha)$ of the following properties. 
\begin{itemize}
\item $f$ is computable in the sense of Weihrauch \cite{We2000}. (We will review the definition in section 2.)
\item ~ $\{ f(x) : x < \beta \}$ is cofinal in $(-\infty, \alpha)$. 
To be more precise, for any $y < \alpha$ there exists $x < \beta$ such that 
$y \leq f(x)$. 
\item $f$ is nondecreasing.
\end{itemize}

A sketch of the proof is as follows. Take computable sequences of rationals such that $\{ a_{n} \} \nearrow \alpha$ and $\{ b_{n} \} \nearrow \beta$. Define points $Q_{n}$ ($n \in \mathbb{N}$) on $\mathbb{R}^2$ as to be $Q_{n} (b_{n}, a_{n})$. We make a line graph by connecting $Q_{n}$ and $Q_{n+1}$ for each $n$. 
Let $f$ be the function whose graph is the line graph. Then $f$ has the properties. 


Suppose we add a requirement that $f$ is Lipschitz continuous to the set of properties mentioned above. It is interesting to note that the extended set of properties is exactly equivalent to the assertion that $\alpha$ is Solovay reducible to $\beta$. We show this equivalence in section 4. Here, Solovay reduction is a reduction between real numbers, and has been deeply studied (\cite[Chapter 9]{DH2010}). Solovay reduction implies Turing reduction, but the converse implication does not hold. Solovay reduction has a connection to the theory of randomness. For example, among left-c.e. reals, completeness with respect to Solovay reduction agrees with 1-randomness. 

We are interested in finding more examples that show the correspondences between 
various continuity concepts and reducibility concepts. 
In classical analysis, H{\"o}lder continuity is one of the well-known continuity concepts. 

\begin{defi}
For $X \subset \mathbb{R}$ and $Y \subset \mathbb{R}$,$f:X \rightarrow Y$ is H{\"o}lder continuous if there exist positive real numbers $H$ and $\xi \leq 1$ such that for any $x_{1}: x_{2} \in X$, the following holds, 
\begin{equation}
|f(x_{1}) - f(x_{2})| < H |x_{1} - x_{2}|^{\xi}
\end{equation}
where the exponent $\xi$ is called the \emph{order}.
\end{defi}

Throughout the paper, unless otherwise specified, we let H{\"o}lder continuity denote that with positive order $\leq 1$. 
In the case where the domain of the function is a closed interval of the real line, 
Lipschitz continuity implies H{\"o}lder continuity,  
and H{\"o}lder continuity implies uniform continuity. 

In the definition of H{\"o}lder continuity, the key quantity is the power of the distance of two real numbers. 
In the context of the study of randomness, 
an important previous work on this quantity is Tadaki's $T$-convergence. 
By developing his study on partial randomness \cite{T2002}, Tadaki investigated convergence by paying attention to $T$th power of difference of real numbers \cite{T2009}, where $T$ is a positive real number $\leq 1$. Another previous work relating to the present paper is the study on partial randomness and Solovay reduction by Miyabe et al. \cite{MNS2018}. 

Our main question is as follows. 
Is there a reducibility that exactly corresponds to H{\"o}lder continuity? 
As far as the authors know, there is no previous work that asks this question as such. 
The answer is affirmative. In this paper, we present such a reduction concept and call it quasi Solovay reducibility. 

We give definitions of  our main notion qS-reducibility and several notions that are necessary to discuss in Section 2. In Section 3, we show several properties of qS-reducibility. Especially, we show qS-reducibility is separated from Solovay reducibility and Turing reducibility. In  Section 4, we discuss relationship of reducibility and continuity. 

\section{Basic definitions and the background}

\subsection{Computable real numbers and computable real functions}

We let $\mathbb{N}, \mathbb{Q}$ and $\mathbb{R}$ denote the set of all natural numbers, rational numbers and real numbers, respectively.
The set of all binary strings of finite length is denoted by $\{ 0,1 \}^{\ast}$. 

In classical recursion theory, unless otherwise specified, a computable function is a computable mapping from natural numbers to natural numbers. 
We are going to investigate computable real functions as well, namely, a computable mapping from real numbers to real numbers. 
In general, we cannot express a real number by a finite string. A name of a real number is given by, for example, a certain infinite sequence of rational numbers. 
Thus in a suitable definition of computable real functions, the usual Turing machine would be insufficient as a model of computation. 
A number of previous studies have prospected suitable definitions. 
Among them, Ko and Friedman \cite{KF1982} introduced a definition based on an oracle Turing machine. In this approach, roughly speaking, computability of a real function $f$ is defined in the following way. An independent real variable $x$ is considered as an oracle function; in other words, an idealized library function. 
We consider the mapping of $0^{n} = 0 \cdots 0$ ($n$ times, where $n$ is a natural number)  to a rational number $q$ approximating $f(x)$ with error at most $2^{-n}$. 
In general, we do not require computability of $x$. 
We require that for each $x$ in the domain of $f$, the above-mentioned mapping of $0^{n}$ to $q$ is computable using $x$ as an oracle, where the algorithm is uniform in the sense that it depends only on $f$ and is independent from $x$. 

A variation of the Ko-Friedman style definition is precisely given in the textbook by 
Weihrauch \cite{We2000}. An equivalent definition is given in \cite{KTZ2018}. 

\begin{defi} \cite{G1957,KF1982,KTZ2018} 
Suppose that $x$ is a real number. 
\begin{enumerate}
\item A \emph{name} of $x$ is a function $\varphi : \{ 0,1 \}^{\ast} \to \{ 0,1 \}^{\ast}$ 
such that for each string $u$, letting $n$ be the length of $u$, $\varphi(u)$ is a binary encoding of an integer $z$ such that $| x - z/2^{n} | \leq 2^{-n}$. 
\item $x$ is a \emph{computable real number} if it has a computable name. 
\end{enumerate}
\end{defi}

Suppose that $x$ is a real number with $0 \leq x \leq 1$. It is well-known that the following (i) and (ii) hold. (i) Any real number $x$ has infinitely many names. 
(ii) $x$ is computable in the above-mentioned sense 
if and only if there is a total computable function $f :\mathbb{N} \to \{ 0,1 \}$ such that 
$0.f(0) f(1) f(2) \cdots$ is a binary expansion of $x$. 
(Proof: (i) For each natural number $n$, there exists a  non-negative integer $v < 2^{n}$ such that $v \leq 2^{n}x \leq v+1$. We may choose either $v$ or $v+1$ as the value of $\varphi (0^{n})$. (ii) $\Leftarrow$ is obvious. If $x$ is rational, $\Rightarrow$ is obvious. Thus in the proof of $\Rightarrow$, we may assume that $x$ is irrational. We can compute $f(0)$ as follows. Find the least natural number $n$ such that either $1/2 < \varphi (n)/2^{n} - 1/2^{n}$ or $\varphi (n)/2^{n} + 1/2^{n} < 1/2$. Since $x \ne 1/2$, we can effectively find such $n$. We let $f(0)=1$ in the former case and $f(0)=0$ in the latter case. In the same way, we can effectively compute $f(1),f(2),\dots$.)

Several types of oracle Turing machines have been discussed in the literature. Here, we consider an oracle Turing machine with a function oracle. 
An \emph{oracle Turing machine} is a Turing machine equipped with a particular read-write tape called an \emph{oracle tape}, and the particular state $q_{\text{query}}$ called the \emph{query state}. 
Suppose that $h$ is a mapping from strings to strings. 
An oracle Turing machine $M$ with an oracle $f$, denoted by $M^{f}$, is similar to the usual Turing machine that computes a function. It gets an input string from a certain tape, and it outputs a string on a certain tape. However, the action of $M^{f}$ differs from those of the usual machine in the following points. When $M$ enters the query state, $M$ writes a string, say $u$, on the oracle tape. The action of writing is regarded as $|u|$ time-steps, where $|u|$ is the length of $u$. Then $u$ is replaced by the string $h(u)$, and $M$ enters another state. This replacement is regarded as one time-step. When $M$ reads $h(u)$,  we count time-steps in the same way as we have done at the writing action. 

An outline of the definition in \cite{We2000} is as follows. This would be sufficient for our purpose. For more rigorous treatment, consult \cite{We2000}.

\begin{defi} \cite{KTZ2018} 
Suppose that $f$ is a partial function from $\mathbb{R}$ to $\mathbb{R}$. 
The function $f$ is \emph{computable} if there exists an oracle Turing machine $M$ with the following property. For each $x$ in the domain of $f$ and for each name $\varphi$ of $x$, $M^{\varphi}$ computes a name of $f(x)$. 
\end{defi}

By saying that \emph{a real function is computable in the sense of Weihrauch}, we mean the above-mentioned sense. 


\subsection{Left-c.e. real numbers and partial randomness}

Unless otherwise specified, 
$\{ a_{n} \} \nearrow \alpha$ denotes that $\{ a_{n} \}$ is a computable nondecreasing sequence of rationals converging to $\alpha$. For a function $f$, the symbol $f(x) \downarrow$ denotes that $f$ is defined on $x$. For a Turing machine $M$, the symbol $M(x) \downarrow$ denotes that $M$ terminates in a finite step for input $x$. 
A set $A$ of natural numbers is computably enumerable, c.e. for short, if there exists a Turing machine $M$ such that $A=\{ n \in \mathbb{N} : M(n) \downarrow \}$. 
The concept of c.e. sets is naturally generalized to a set of rational numbers, or to a set of binary strings, under a certain coding.

Although the concept of computable real numbers is quite natural, 
the class of all computable real numbers is too narrow in the field of algorithmic randomness, for any computable real number is not random. 
A fruitful relaxed class is the class of all left-c.e real numbers. 
As we defined in the Introduction section, a real number is left-c.e. if the left cut is computably enumerable. 
A typical example of a random real number is Chaitin's $\Omega$. 
For the detailed definitions of Chaitin's $\Omega$ and related concepts, consult standard textbooks of algorithmic randomness such as \cite{DH2010} and \cite{N2009}. 
It is well known that Chaitin's $\Omega$ is left-c.e. 


Turing reducibility is a concept that compares the complexity of two functions $f, g : \mathbb{N} \to \{ 0, 1 \}$. $f$ is \emph{Turing reducible} to $g$ ($f \leq_{T} g$ in symbol) if there exists an oracle Turing machine $M$ such that for each natural number $n$, $M^{g}$ halts for input $n$ and outputs $f(n)$. 
Turing reduction between two sets of natural numbers means Turing reduction between their characteristic functions. 
For real numbers $\alpha$ and $\beta$, let $0.f(1)f(2)\cdots$ and $0.g(1)g(2)\cdots$ be their binary expansions such that in each expansion, 0 has infinitely many occurrences. 
It is easy to see that $f$ is Turing reducible to $g$ if and only if the left cut of $\alpha$ is Turing reducible to the left cut of $\beta$. 
Turing reduction between two real numbers is defined as this meaning, that is, Turing reduction between their left cuts. If $\alpha \leq_{T} \beta$, we also say $\alpha$ is \emph{$\beta$-computable}. 

In the theory of computing, not a few variations of Turing reducibility have been introduced. 
 Interestingly enough, many important reducibility concepts have common properties that are known as standard in the following sense. 

\begin{defi} \label{defi:std-red}
\cite[Chapter 9]{DH2010} The reducibility $r$ is \emph{standard} if the following hold.
\begin{enumerate}
\item $r$ is $\Sigma_{3}^{0}$.
\item Every computable real is reducible to any given left-c.e.real. 
\item Real addition is a join in the r-degrees of left-c.e.reals.
\item For any left-c.e. real $\alpha$ and any rational $q > 0$, we have $\alpha \equiv_{r} q\alpha$.
\end{enumerate}
The condition (3) of the above definition means that the least upper bound of the r-degrees of the real $\alpha$ and r-degree of the real $\beta$ is given by the r-degree of the real $\alpha + \beta$.
\end{defi}

Turing reduction is of course a standard reducibility. 
Another important example of standard reducibilities is Solovay reducibility. 

\begin{defi} \label{defi:S-qS}
Suppose that $\alpha$ and $\beta$ are real numbers. 
\begin{enumerate}
\item \cite[Chapter 9]{DH2010} $\alpha \leq_{S} \beta$ ($\alpha$ is \emph{Solovay reducible} to $\beta$) if there exist a partial computable function $f$ from $\mathbb{Q}$ to $\mathbb{Q}$ and a positive natural number $d$ with the following property. For any rational $x < \beta$, we have $f(x) \downarrow < \alpha $ and 
$\alpha - f(x) < d (\beta - x)$.
\item $\alpha \leq_{qS} \beta$ ($\alpha$ is \emph{quasi Solovay reducible} to $\beta$) if there exist a partial computable function $f$ from $\mathbb{Q}$ to $\mathbb{Q}$ and positive natural numbers $d,\ell$ with the following property. For any rational $x < \beta$, we have $f(x) \downarrow < \alpha $ and 
and $(\alpha - f(x))^{\ell} < d (\beta - x)$ 
(in other words, $(\alpha - f(x)) < H (\beta - x)^{1/\ell}$, where $H=d^{1/\ell}$). 
\end{enumerate}
\end{defi}

Even if an infinite binary sequence $\alpha = a_{0} ~ a_{1} ~ a_{2} ~ a_{3} \cdots$ is Martin-L\"{o}f random, $\beta = a_{0} ~ 0 ~ a_{1} ~ 0 ~ a_{2} ~ 0 ~ a_{3} \cdots$ is not Martin-L\"{o}f random. However, it is the natural to regard $\beta$ as a partial random sequence.  Some important investigations in the earlier stage of partial randomness are in \cite{T2002}. A short summary of the research in this line may be found in \cite{St2005}. Some generalizations of partial randomness concepts has been discussed in \cite{HHSY2013}. 
Here, we review the following terminology from \cite{T2002}. 

\begin{defi}
\cite{T2002} Let $T \in (0,1]$ be a real number. A Real number $\alpha$ is \emph{weakly Chaitin $T$-random} if 
\begin{equation}
\forall n \in \mathbb{N}^{+}[ Tn \leq^{+} K(\alpha \upharpoonright_n) ]
\end{equation}
\end{defi}

\begin{defi}
\cite{T2002} Let $T \in (0,1]$ be a real number. A Real number $\alpha$ is \emph{$T$-compressible} if 
\begin{equation}
K(\alpha \upharpoonright_n) \leq Tn + o(n)
\end{equation}
\end{defi}

\subsection{Our notation}

Here, we introduce new symbols. 

\begin{defi} Suppose that $\alpha$ and $\beta$ are real numbers. 
\begin{enumerate}
\item $(L)_{1}$ denotes the assertion that there exists a function $f :(-\infty, \beta) \to (-\infty, \alpha)$ of the following properties.
\begin{enumerate}
\item $f$ is computable in the sense of Weihrauch \cite{We2000}. 
\item $f$ is Lipschitz continuous. To be more precise, there exists a positive real number $L$ such that for any $x_{1}, x_{2} < \beta$, 
$|f(x_{1}) - f(x_{2})| < L |x_{1} - x_{2}|$. 
\item ~ $\{ f(x) : x < \beta \}$ is cofinal in $(-\infty, \alpha)$. 
To be more precise, for any $y < \alpha$ there exists $x < \beta$ such that 
$y \leq f(x)$. 
\item $f$ is nondecreasing.
\end{enumerate}
\item Suppose that 
there exists a function $f :(-\infty, \beta) \to (-\infty, \alpha)$ satisfying (a), (b), (c) above and in addition (e) of the following. 

\medskip

(e) There exists a strict increasing sequence of rationals $\{ r_{n} \}$ such that $r_{n} \to \beta - 0$ and $f(r_{n}) \to \alpha - 0$. Here, the sequence $\{ r_{n} \}$ may be non-computable.

\medskip

Then the resulting assertion is denoted by $(L)_{2}$. 
Note that $(L)_{2}$ is equivalent to (a) + (b) + (e). 
\item Suppose that 
there exists a function $f :(-\infty, \beta) \to (-\infty, \alpha)$ satisfying (a), (c), (d)  and in addition (bH) of the following. 

\medskip

(bH) $f$ is H{\"o}lder continuous with the positive order $< 1$. To be more precise, there exists a positive real numbers $H$ and $\xi$ such that $\xi < 1$ and for any $x_{1}, x_{2} < \beta$, 
$|f(x_{1}) - f(x_{2})| < H |x_{1} - x_{2}|^{\xi}$. 

\medskip

Then the resulting assertion is denoted by $(H)_{1}$.
\item Suppose that 
there exists a function $f :(-\infty, \beta) \to (-\infty, \alpha)$ satisfying (a), (bH), (c) and (e). 
Then the resulting assertion is denoted by $(H)_{2}$. 
Note that $(H)_{2}$ is equivalent to (a) + (bH) + (e).
\end{enumerate}
\end{defi}
A real number in the unit interval is, by taking its binary expansion in which 0 has infinitely many occurrences, often identified with an infinite binary sequence. In this case, we need to be aware of the following point. 
Suppose that we know that $|\beta - \alpha|$, the geometrical distance of two real numbers $\alpha$ and $\beta$ in the real line, is at most $2^{-n}$. In general, this assumption does not imply agreement between the first $n$ bits of $\alpha$ and those of $\beta$ (their binary expansions in the above-mentioned style). 
For example, observe the case where $\alpha = 0.0011111\dot{0}$, $\beta =0.01\dot{0}$ and $n=7$. 
This delicate relationship between geometrical distance and agreement of bits 
sometimes appear as an obstacle in the study of Solovay reduction. In order to avoid such an obstacle, we introduce the following two sets of numbers, $\mathbb{D}^{\prime}$ and $\mathbb{R}^{\prime}$, as follows. 
\begin{defi}
$\mathbb{D}'$ is the set of all rationals $q~(0 < q < 1)$ with the following properties. 
\begin{enumerate}
\item $q$ is a dyadic rational (a dyadic rational is a rational number of the form $q = z2^{-n}$ where $z \in \mathbb{Z}$ and $n \in \mathbb{N}$ \cite[Chapter 1]{N2009}).
\item $q$ has binary expansion $q=0.q_1q_2 \cdots q_{2k-1}q_{2k}$, where for each $i(=1,\cdots,k)$, $q_{2i} = 1 - q_{2i-1}$
\end{enumerate}
For each element $q \in \mathbb{D}'$, let $k(q)$ denote the above-mentioned $k$. 
\end{defi}
For example, $0.1010$ is in $\mathbb{D}'$ where $k(0.1010) = 2$. On the other hand, $0.1110$ is not in $\mathbb{D}'$.
\begin{defi}
$\mathbb{R}'$ is the set of all reals $\beta$ with the following properties. 
\begin{enumerate}
\item $\beta$ is not rational.
\item $\beta$ has a binary expansion $\beta = 0.b_1b_2\cdots$ with the following property. 
\begin{equation}
\forall n \in \mathbb{N}^{+} [ 0.b_1b_2 \cdots b_{2n-1}b_{2n} \in \mathbb{D}' ].
\end{equation}
\end{enumerate}
\end{defi}

\section{The relationships among the reducibilities}

There are some known characterizations of Solovay reducibility by means of sequences. 
A characterization by Calude et al. \cite{CCHK1999} (see also \cite[Proposition 9.12]{DH2010}) 
can be generalized to quasi Solovay reduction as follows. 

\begin{lemm} Suppose that $\alpha$ and $\beta$ are left-c.e. reals. Then the following are equivalent. 
\begin{enumerate}
\item $\alpha \leq_{qS} \beta$.
\item For every $\{ a_{n} \}\nearrow \alpha$ and $\{ b_{n} \}\nearrow \beta$, 
there exist an increasing computable function $g :\mathbb{N} \to \mathbb{N}$ and positive integers $d$ and $\ell$ such that for each $n \in \mathbb{N}$, the following holds. 
\begin{equation}
(\alpha - a_{g(n)})^{\ell} \leq d(\beta - b_n).
\end{equation}
\item For every $\{ b_{n} \}\nearrow \beta$, there exist $\{ a_{n} \}\nearrow \alpha$ and positive integers $d$ and $\ell$ such that for each $n \in \mathbb{N}$, the following holds.
\begin{equation}
(\alpha - a_n)^{\ell} \leq d(\beta - b_n).
\end{equation}
\item There exist $\{ a_{n} \}\nearrow \alpha$, $\{ b_{n} \}\nearrow \beta$ and positive integers $d$ and $\ell$ such that for each $n \in \mathbb{N}$, the following holds. 
\begin{equation}
(\alpha - a_n)^{\ell} \leq d(\beta - b_n).
\end{equation}
\end{enumerate}
\end{lemm}
\begin{proof}
Proof of (1) $\Rightarrow$ (2): ~ 
Let $f: \mathbb{Q} \to \mathbb{Q}$, $d$ and $\ell$ be witnesses of $\alpha \leq_{qS} \beta$. 
We are going to define a mapping $g :\mathbb{N} \to \mathbb{N}$ by means of recursion. 
Given $n \in \mathbb{N}$, find the least $s \in \mathbb{N}$ such that 
$f(b_n) < a_s$ and $s > g(i)$ (for all $i < n$). Then we define $g(n)$ as to be this $s$. 

Proof of (2) $\Rightarrow$ (3): ~ 
Since $\alpha$ is left-c.e., there exists a sequence $\{ a_{n} \} \nearrow \alpha$. 
Take a $g :\mathbb{N} \to \mathbb{N}$ in the statement of assertion 2. 
The sequence $\{ a_{g(n)} \}$ is what we want. 

(3) $\Rightarrow$ (4) is obvious.

Proof of (4) $\Rightarrow$ (1): ~ 
Let $\{ a_{n} \}\nearrow \alpha$, $\{ b_{n} \}\nearrow \beta$, $d$ and $\ell$ be witnesses of assertion 4. For each rational number $q < \beta$, find a natural number $n$ such that 
$b_n \leq q < b_{n+1}$. Define $f(q)$ as to be $a_{n+1}$. 
Then it holds that $(\alpha - f(q))^{\ell} < d (\beta - q)$.
\end{proof}

\begin{lemm} Suppose that $\leq_{qS}$ is the relation of left-c.e.reals. 
\begin{enumerate}
\item $\leq_{qS}$ is a pseudo order.
\item $\leq_{qS}$ is a standard reducibility. 
\end{enumerate}
\end{lemm}

\begin{proof}
(1) 
(Reflexivity)\\
For each real $\alpha$,we set $\ell$=1, d=2, $f(x)=x$ in the definition of $\leq_{qS}$.\\
(Transitivity)\\
Suppose that $\alpha \leq_{qS} \beta$ holds with witness $\ell_1$,$d_1$,$f$ and that $\beta \leq_{qS} \gamma$ holds with witness $\ell_2$,$d_2$,$g$. We define $h=f \circ g$. For each rational $q<\gamma$, $g(q)\downarrow < \beta$ and $h(q)=f \circ g(q)\downarrow<\alpha$.\\
 Now we have $(\alpha - h(q))^{\ell_1} < d_1(\beta - g(q))$ and $(\beta - g(q))^{\ell_2} < d_2(\gamma - q)$ for each rational $q$. So we have $(\alpha - h(q))^{\ell_1\ell_2}<d_{1}^{\ell_2}(\beta - g(q))^{\ell_2}<d_{1}^{\ell_2}d_2(\gamma - q)$ for each rational $q < \gamma$.

(2) The conditions (1),(2) and (4) of Definition~\ref{defi:std-red} hold for $\leq_{qS}$ because $\leq_{S}$ is a standard reducibility.We prove the following proposition for (3) of Definition~\ref{defi:std-red}.

Claim: Suppose that $\alpha$ and $\beta$ are left-c.e. reals. Then we have the following.
\begin{equation}
\deg_{qS}(\alpha + \beta) = \sup\{\deg_{qS}\alpha,\deg_{qS}\beta\}
\end{equation}
Here, the supremum is taken among left-c.e. reals.

Proof of the claim. 

It suffices to show the following two propositions.

(i) $\alpha \leq_{qS} \alpha + \beta$, $\beta \leq_{qS} \alpha + \beta$.

(ii) For a left-c.e.real $\gamma$, $\alpha \leq_{qS} \gamma$, $\beta \leq_{qS} \gamma$ $\Rightarrow \alpha + \beta \leq_{qS} \gamma$.

Proof of (i)

We have $\alpha \leq_{S} \alpha + \beta$ by Definition~\ref{defi:S-qS} and $\leq_{S}$ implies $\leq_{qS}$. Hence $\alpha \leq_{qS} \alpha + \beta$. We have $\beta \leq_{qS} \alpha + \beta$ by the same argument as above.

Proof of (ii)

Suppose that $\gamma$ is a left-c.e.real and we have $\alpha \leq_{qS} \gamma$ and $\beta \leq_{qS} \gamma$. Let $\langle f_{0}, c_{0}, \ell_{0} \rangle$ and $\langle f_{1}, c_{1}, \ell_{1} \rangle$ be witnesses, respectively. If $q \in \mathbb{Q}$ and $q < \gamma$ then $f_{0}(q)\downarrow < \alpha$,$(\alpha - f_{0}(q))^{\ell_{0}} < c_{0}(\gamma - q)$ and $f_{1}(q)\downarrow < \beta$,$(\beta - f_{1}(q))^{\ell_{1}} < c_{1}(\gamma - q)$. We can assume $\ell_{0} < \ell_{1}, c_{0} \geq 1, c_{1} \geq 1$ and $\gamma -q <1$ without loss of generality. We set $f_{2}(x) = f_{0}(x) + f_{1}(x)$. Then it holds that $f_{2}(q)\downarrow < \alpha + \beta$, and we have the following.
\begin{align}
(\alpha + \beta) - f_{2}(q) &< c_{0}^{1/\ell_{0}}(\gamma - q)^{1/\ell_{0}} + c_{1}^{1/\ell_{1}}(\gamma - q)^{1/\ell_{1}} \notag \\
&\leq c_{0}^{1/\ell_{0}}(\gamma - q)^{1/\ell_{1}} + c_{1}^{1/\ell_{0}}(\gamma - q)^{1/\ell_{1}} \notag \\
&= (c_{0}^{1/\ell_{0}} + c_{1}^{1/\ell_{0}})(\gamma - q)^{1/\ell_{1}}
\end{align}
We set $c_{2} : = (c_{0}^{1/\ell_{0}} + c_{1}^{1/\ell_{0}})^{\ell_{1}}$ then $((\alpha + \beta) - f_{2}(q))^{\ell_{1}} \leq c_{2}(\gamma - q)$. 
Therefore $\alpha + \beta \leq_{qS} \gamma$ via $f_{2},c_{2},\ell_{1}$. 
\end{proof}
\begin{lemm} \label{lemm:S-qS-sep}Suppose that $\alpha$ and $\beta$ are left-c.e. reals. 
\begin{enumerate}
\item $\alpha \leq_{S} \beta $ implies $\alpha \leq_{qS} \beta $. 
\item $\alpha \leq_{qS} \beta $ does not imply $\alpha \leq_{S} \beta $. 
\end{enumerate}
\end{lemm}
\begin{proof}
(1) follows by setting $\ell$ in the definition to be 1. 
\\ 
(2)
Claim 1: Suppose that $q \in \mathbb{D}'$, $\beta \in \mathbb{R}'$, $q < \beta$ and that $|\beta - q| \leq 2^{-(2m+1)}$ for some $m \in \mathbb{N}$. Then, the following holds.
\begin{equation}
\forall i \in \mathbb{N} [1 \leq i \leq 2m \Rightarrow b_i = q_i ]
\end{equation} 

Proof of Claim 1.

We prove this claim by induction.
Base step $m=0$: the claim is obvious.

Induction step $m = s+1$: Our induction hypothesis is as follows. 

\begin{equation}
1 \leq i \leq 2s \Rightarrow b_i = q_i
\end{equation}

Case 1: $b_{2s+1}b_{2s+2} = 01$. 

If $q_{2s+1}q_{2s+2} = 10$ then $\beta < 0.b_1 \cdots b_{2s}0111$ and $q \geq 0.b_1 \cdots b_{2s}10$. Hence $q > \beta$. We have a contradiction.

If $q_{2s+1}q_{2s+2} = 00$(This may well happen in the case of $k(q) < s$) then $\beta > 0.b_1 \cdots b_{2s}0101$ and $q = 0.b_1 \cdots b_{2s}$. Hence $\beta - q > 0.\underbrace{0 \cdots 0}_{2s}0101 > 0.\underbrace{0 \cdots 0}_{2s}001 = 2^{-(2m+1)}$. We have a contradiction. Hence $q_{2s+1}q_{2s+2} = 01$ holds. Therefore the following holds.

\begin{equation}
1 \leq i \leq 2s+2 \Rightarrow b_i = q_i
\end{equation}

Case 2: Otherwise. In other words $b_{2s+1}b_{2s+2} = 10$.

If $q_{2s+1}q_{2s+2} = 01$ then $\beta > 0.b_1 \cdots b_{2s}1001$ and $q \leq 0.b_1 \cdots b_{2s}0111$. Hence $\beta - q > 0.\underbrace{0 \cdots 0}_{2s}001 = 2^{-(2m+1)}$. We have contradiction.

If $q_{2s+1}q_{2s+2} = 00$ we can derive contradiction by the same argument as above. Hence $q_{2s+1}q_{2s+2} = 10$ holds. Therefore the following holds.

\begin{equation}
1 \leq i \leq 2s+2 \Rightarrow b_i = q_i
\end{equation}

Claim 1 is proved.

We are going to show the existence of left-c.e.reals $\alpha$ and $\beta$ such that $\alpha \not\leq_{S} \beta$ and $\alpha \leq_{qS} \beta$.
Given $\alpha$, let $0.\alpha (0) \alpha (1) \cdots$ be its binary
expansion that has infinitely many occurrences of 0. Thus $\alpha (n)$
is the $n+1$st decimal place. For each $n \in \mathbb{N}$, we define $h_1(\alpha)(2n) := \alpha(n)$ and $h_1(\alpha)(2n+1) := 1-\alpha(n)$.

Claim 2: 
\begin{enumerate}
\item $\alpha$ is left-c.e. $\Rightarrow$ $h_1(\alpha)$ is left-c.e.
\item $\alpha$ and $h_1(\alpha)$ are left-c.e. $\Rightarrow \alpha \leq_{qS} h_1(\alpha)$.
\end{enumerate}
Proof of (1) of Claim 2.\\
If $\alpha$ is a left-c.e.real then there exists a computable increasing sequence of rationals $\{a_n\}_{n \in \mathbb{N}}$ converging to $\alpha$. We can assume that $a_{n} \in \mathbb{D}$ for all $n \in \mathbb{N}$.Then $a_{n}=\sum_{i=1}^{k} a_{i}^{(n)}2^{-i}$($a_{i}^{(n)}=0$ or 1) for some $k$. We define a computable sequence of rationals $\{b_n\}_{n \in \mathbb{N}}$ such that $b_{n}=\sum_{i=1}^{k} (a_{i}^{(n)} + 1)4^{-i}$ for each $n \in \mathbb{N}$. The sequence is increasing and converges to $h_{1}(\alpha)$.\\
Proof of (2) of Claim 2.\\
Let $\alpha$ be a left-c.e.real and $\beta = h_1(\alpha)$. We are going to prove that there exists a partial computable function $f$ with the following property.\\
\begin{equation}
\forall q \in \mathbb{Q}[\beta - 2^{-5} < q < \beta \Rightarrow ( f(q)\downarrow < \alpha \land (\alpha - f(q))^{4} < \beta - q)]
\end{equation}
Definition of $f$: Given rational number $q < \beta$, we can effectively find $q'$ such that $q \leq q' < \beta$ and $q' \in \mathbb{D}'$ because $\beta = h_1(x)$ is in $\mathbb{R}'$. 
For $q' = 0.q'_{1}q'_{2} \cdots q'_{2k-1}q'_{2k}  (k=k(q))$, we define $f(q) := 0.q'_{1}q'_{3} \cdots q'_{2k-1}$.\\
Verification: Case 1. If there exists $m \in \mathbb{N}$ such that $m \geq 2$ and $2^{-(2m+2)} < \beta - q' \leq 2^{-(2m+1)}$ then by Claim 1, $\alpha$ and $f(q)$ have binary expansions whose first $m$ bits coincide. Hence $|\alpha - f(q)| \leq 2^{-m}$ and $|\alpha - f(q)|^4 \leq 2^{-4m} \leq 2^{-(2m+2)} < |\beta - q'| \leq |\beta - q|$. In other words, $(\alpha - f(q))^4 < \beta - q$. 
\\
Case 2. Otherwise. There exists $m \in \mathbb{N}$ such that $m \geq 2$ and $2^{-(2m+3)} < \beta - q' \leq 2^{-(2m+2)}$. Then by Claim 1, $\alpha$ and $f(q)$ have binary expansions whose first $m$ bits coincide. Hence $|\alpha - f(q)| \leq 2^{-m}$ and $|\alpha - f(q)|^4 \leq 2^{-4m} \leq 2^{-(2m+3)} < |\beta - q'| \leq |\beta - q|$. In other words, $(\alpha - f(q))^4 < \beta - q$. Therefore $\alpha \leq_{qS} \beta$. Claim 2 is proved.\\ 
Let $\alpha = \Omega$ and $\beta = h_1(\alpha)$. $\alpha \not\leq_{S} \beta$ because $\alpha$ is 1-random and $\beta$ is not 1-random ( Kolmogorov complexities of $\beta$ are small). Therefore we have left-c.e.reals $\alpha$ and $\beta$ such that $\alpha \not\leq_{S} \beta$ and $\alpha \leq_{qS} \beta$.\\
\end{proof}

\begin{lemm} \label{lemm:qS-T-sep}
Suppose that $\alpha$ and $\beta$ are left-c.e. reals. 
\begin{enumerate}
\item $\alpha \leq_{qS} \beta $ implies $\alpha \leq_{T} \beta $. 
\item $\alpha \leq_{T} \beta $ does not imply $\alpha \leq_{qS} \beta $. 
\end{enumerate}
\end{lemm}

\begin{proof}
(1) 
Suppose that $\langle f, d, \ell \rangle$ is a witness of the qS reduction. 
Take a $\beta$-computable sequence $\{ \gamma_{n} \} \nearrow \beta$ of the following property. 
\begin{equation}
\forall n ~ \beta - \gamma_{n} \leq 2^{-\ell n}
\end{equation}

Then we have the following. 

\begin{equation}
\alpha - f (\gamma_{n}) \leq d^{1/\ell} ( \beta - \gamma_{n} )^{1/\ell} 
\leq d^{1/\ell} 2^{-n}
\end{equation}

Hence, $\alpha $ is $\beta$-computable. 

(2) 
We are going to show the existence of reals $\alpha$ and $\beta$ such that $\alpha \leq_{T} \beta$ and  $\alpha \not\leq_{qS} \beta$. 
Let  $\Omega = 0.\alpha_1\alpha_2\alpha_3\cdots \alpha_n\cdots$ be the binary expansion of Chaitin's halting probability $\Omega$. 
Let $\beta$ be the real number whose binary expansion is given as follows. 
\begin{equation}
\beta = 0.\alpha_1\alpha_2\alpha_2\alpha_3\alpha_3\alpha_3\cdots \underbrace{\alpha_n \cdots \alpha_n}_{n} \cdots
\end{equation}
Then $\Omega \equiv_{T} \beta$. 
Assume for a contradiction that $\Omega \leq_{qS} \beta$. Then there exist $\ell,k \in \mathbb{N}$ and a partial computable function $f:\mathbb{Q} \rightarrow \mathbb{Q}$ such that for all rationals $q < \beta $, $f(q)$ is defined and $(\Omega - f(q))^{\ell} < 2^k(\beta - q)$. 
For each bit string $\sigma = x_{1} \cdots x_{m}$, we define a rational $q(\sigma)$ as to be $0.x_1 x_2 x_2 x_3 x_3 x_3 \cdots \underbrace{x_m \cdots x_m}_{m}$. 
In particular, in the case where $\sigma$ is the first $m$ bits $\alpha_{1} \cdots \alpha_{m}$ of $\Omega$, $q=q(\sigma)$ is the first $m(m+1)/2$ bits of $\beta$. 
In this case we have $(\Omega - f(q))^{\ell} < 2^{k}(\beta - q) < 2^{k - m(m+1)/2}$, 
thus $\Omega - f(q) < 2^{- (m^2 - 2k)/(2\ell)}$. 
Therefore $\Omega $ belongs to the following interval: 
$I_\sigma = ( f(q) - 2^{- (m^2 - 2k)/(2\ell)}, f(q) + 2^{- (m^2 - 2k)/(2\ell)} )$. 

By means of these intervals, we construct a Martin-L{\"o}f test. 
We define $U_m$ as to be the union of all $I_{\sigma}$ among all binary strings $\sigma$ of length $m$ such that $f(q(\sigma))$ is defined. 
Then $\{ U_{m} \}_{m}$ is uniformly $\Sigma_{1}$. In addition, 
$\lambda(U_m) \leq 2^m \cdot 2 \cdot2^{- (m^2 - 2k)/(2\ell)}$. 
Thus for all but finitely many $m$, it holds that $\lambda(U_m) \leq 2^{-m}$. 
By the previous paragraph, $\Omega$ belongs to $U_m$ for all $m$. 
This contradicts to the fact that $\Omega$ is Martin-L{\"o}f random.
\end{proof}

\begin{lemm} Suppose that $\alpha$ and $\beta$ are left-c.e.reals and $\alpha \in \mathbb{R}^{'}$. If $\alpha \leq_{qS} \beta$ with witness $<f,c,\ell>$, where $\beta \upharpoonright_{( \ell \times n )} $ $=\sigma_{1}\frown \cdots \frown \sigma_{\ell}$, 
$| \sigma_{i}|=n$ then the following holds. 
\begin{equation}
K (\alpha \upharpoonright_n) / \ell \leq 
\max_{i} \{ K (\sigma_{i}) \} + O(1)
\end{equation}
\end{lemm}

\begin{proof}
Suppose that $\alpha \in \mathbb{R}^{'}$ and $\alpha \leq_{qS} \beta$ with witness $<f,c,\ell>$. We can compute $\alpha \upharpoonright_n$ from $\beta \upharpoonright_{( \ell \times n )} $ and little more constant bits. Therefore, 
\begin{align}
K (\alpha \upharpoonright_n) &\leq K (\beta \upharpoonright_{( \ell \times n )}) + O(1) \\ \notag
&= K (\sigma_{1}\frown \cdots \frown \sigma_{\ell}) + O(1) \\ \notag
&\leq K (\sigma_{1}) + \cdots + K (\sigma_{\ell}) + O(1) \\ \notag
&\leq \ell \max_{i} \{ K (\sigma_{i}) \} + O(1)
\end{align}
The lemma has been proved.
\end{proof}

In the study of partial randomness, Tadaki \cite{T2002} introduced the concept of 
\emph{generalized halting probability} $\Omega^{T}$ for each positive real number 
$T \leq 1$. 

\begin{equation}
\Omega^{T} := \sum_{p \in \mathrm{dom} U} 2^{-|p|/T}
\end{equation}

Here, $U$ is a fixed universal prefix-free machine. In the case of $T=1$, $\Omega^{1}$ coincides with the usual Chaitin's halting probability $\Omega$. 
In particular, Tadaki showed that $\Omega^{T}$ is weakly Chaitin $T$-random and $T$-compressible. 

In the case of $T=2^{-n}$ and $n$ is a natural number, we introduce a \emph{modified generalized halting probability} $\Omega_{T}$. Let $h_1$ be the function defined in the proof of Lemma~\ref{lemm:S-qS-sep} (2).  

\begin{equation}
\Omega_{2^{0}} := \Omega, \quad \Omega_{2^{-(n+1)}} := h_{1} (\Omega_{2^{-n}})
\end{equation}

\begin{lemm} 
Let $n \in \mathbb{N}$ and $T=2^{-n}$.
\begin{enumerate}
\item $\Omega_{T}$ is a left-c.e. real number. 
\item $\Omega_{T}$ is qS-complete among left-c.e.reals.
\item $\Omega_{T}$ is weakly Chaitin $T$-random.
\item $\Omega_{T}$ is $T$-compressible.
\end{enumerate}
\end{lemm}

\begin{proof}
The case of $n=0$ is well known fact. 
We are going to prove the case of $n=1$. The assertion (1) is immediate from Claim 2 (1) in the proof of Lemma~\ref{lemm:S-qS-sep} (2) and the fact that $\Omega$ is left-c.e. 
The assertion (2) holds because $\Omega$ is S-complete and $\Omega \leq_{qS} h_{1} (\Omega)$ by Claim 2 (2). For all $k \in \mathbb{N}$, we have $K(\Omega\upharpoonright_k) \leq_{+} K(h_1(\Omega)\upharpoonright_{2k}) \leq_{+} K(\Omega\upharpoonright_k)$. Hence $k \leq_{+} K(h_1(\Omega)\upharpoonright_{2k}) \leq_{+} k + 2\log{k}$ (see \cite[Section2.2]{N2009}). If $k$ is even then $k/2 \leq_{+} K(h_1(\Omega)\upharpoonright_{k}) \leq_{+} k/2 + 2\log{k/2} = k/2 +o(k)$. If $k$ is odd then the complexity differs from the even-case at most up to  a constant. Therefore the assertions (3) and (4) hold for $n = 1$. The induction step is shown in the same way as the above-mentioned case of $n=1$. Thus, the assertions hold for all $n \in \mathbb{N}$.
\end{proof}
\section{The reductions and notions of continuity}

\begin{theo} \label{theo:lip1}
Suppose that $\alpha$ and $\beta$ are left-c.e. reals. 
Then, the following three assertions are equivalent: 
$(L)_{1}$, $(L)_{2}$, and ``$\alpha \leq_{S} \beta$''. 
\end{theo}

\begin{proof}
$(L)_{1} \Rightarrow (L)_{2}$ is obvious. 


$(L)_{2} \Rightarrow \alpha \leq_{S} \beta$: 
Assume that $(L)_{2}$ holds with witnesses $f, L$ and $\{ r_{n} \}$. 
Given a $q \in (-\infty, \beta) \cap \mathbb{Q}$, we are going to choose a rational number $g(q)$. The value $g(q)$ will be an approximate value of $f(q)$. 
The function $g$ and $L$ will be the witnesses of $\alpha \leq_{S} \beta$.

Let $f(q)=0.s_{1}s_{2}\cdots$ be a binary expansion where 0 has infinitely many occurrences.  
Since $f$ is computable in the sense of Weihrauch and $q$ is a rational, 
$f(q)$ is a computable real. In particular, the mapping of 
$n \mapsto s_{n}$ is a total computable function. 
For each $n \geq 1$, let $k=k(n)$ be the least $k \geq n$ such that $s_{n}=0$. 
Then let $c_{n}=0.s_{1}\cdots s_{k-1}1$. 
Each $c_{n}$ is a rational, and $c_{n} \to f(q)+0$. 

Since $\alpha$ is a left-c.e. real, its left set $W_{a}=\{ r\in \mathbb{Q} : r < \alpha \}$ is c.e. Therefore, by means of parallel search with respect to $n$, we can effectively find a natural number $n$ such that $c_{n} \in W_{a}$. Let $m$ be such a number that we first find. Let $g(q)=c_{m}$. This completes the definition of $g$. 

Now we are going to verify that $g$ and $L$ are the witnesses of the Solovay reduction. 
For $n$ that is large enough, it holds that $g(q) \leq f(r_{n})$. Thus we have $f(q) \leq c_{m} = g(q) \leq f(r_{n})$. Therefore, the following hold for all but finitely many natural numbers $n$. 
\begin{equation}
|f(r_{n}) -g(q)| \leq |f(r_{n})-f(q)| \leq L |r_{n}-q|
\end{equation}

The last inequality holds by the Lipschitz continuity of $f$. 
By taking the limit of $n \to \infty$, we have $|\alpha - g(q)| \leq L|\beta - q|$. 
Hence, it holds that $\alpha \leq_{S} \beta$. 


$\alpha \leq_{S} \beta \Rightarrow (L)_{1}$: 
Suppose $\alpha \leq_{S} \beta $. There exist a partial computable function $f$ and a positive integer $d$ with the following properties. 
For each rational $q < \beta$, we have $f(q) \downarrow < \alpha$ and $\alpha - f(q) < d (\beta - q)$. 

For each point $(w,z)$ such that $w \leq \beta$ and $z \leq \alpha$, 
we define closed region $D_{w,z}$ as follows. 

\begin{equation}
D_{w,z} =
\{ (x,y) \in \mathbb{R}^{2} : -d(w-x) +z \leq y \leq z \}
\end{equation}

Let $\{ b_{n} \}_{n}$ be a computable sequence of rationals that increasingly converges to $\beta$. 
Without loss of generality, we may assume that for each $n$ it holds that 
$f(b_{n}) < f(b_{n+1})$. 

Let $Q_{0}$ be the point $(b_{0}, f(b_{0}))$. 
By our assumption on $f$ and $d$, 
it holds that $f(b_{0}) < \alpha$, and $\alpha - f(b_{0}) < d (\beta - b_{0})$. 
By means of the last inequality, 
we have $- d (\beta - b_{0}) + \alpha < f(b_{0})$. 
Hence $Q_{0}$ is an interior point of $D_{\beta,\alpha}$. 
Therefore, for any $x \leq b_{0}$ the point $(x, f(b_{0}))$ is an interior point of $D_{\beta, \alpha}$. 

For each $x \leq b_{0}$ we define $g(x)$ as to be $f(b_{0})$. 
Now suppose that $n$ is a natural number and we have defined interior points $Q_{i}$ ($i=0, \dots, n$) of $D_{\beta,\alpha}$. By connecting $Q_{i}$s we get a line graph, and we define $g(x)$ for $x \leq b_{n}$ by this line graph. 
In Figure~\ref{fig:graph_solo2lip}, the line leading to $(\beta, \alpha)$ denotes the line $y = \alpha, x \leq \beta$. The diagonal line starting from $(\beta, \alpha)$ denotes the line $y=-d(\beta -x) + \alpha $. The region between the two lines is $D_{\beta, \alpha}$. The thick half line of Figure~\ref{fig:graph_solo2lip} is the set $\{ (x, f(b_{0})) : x \leq b_{0 }\}$.  

\begin{figure}[htb] 
\centering
\includegraphics[width=.5\textwidth]{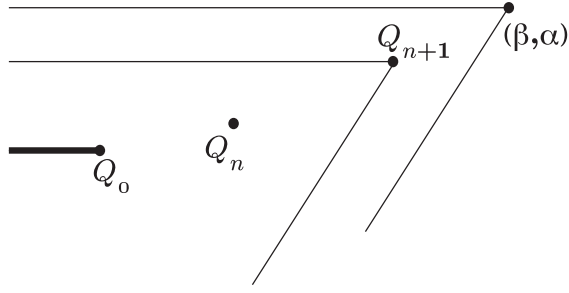}
\caption{Construction of the graph of $g$\label{fig:graph_solo2lip}}
\end{figure}

For sufficiently large $N$, all the $Q_{i}$ ($i=0, \dots, n$) are 
interior points of the region $D_{b_{N}, f(b_{N})}$. 
We can effectively find one of such $N$. Let $Q_{n+1}$ be 
the point $(b_{N}, f(b_{N}))$. $Q_{n+1}$ is an interior point of $D_{\beta,\alpha}$ in the same manner as $Q_{0}$. 
This completes the inductive definition of $g : (-\infty,\beta) \to (-\infty, \alpha)$. 

Then $g$ is computable in the sense of Weihrauch, and if $x \to \beta - 0$ then $g(x) \to \alpha - 0$. In addition, $g$ is increasing. We are going to show that $g$ is Lipchitz continuous with Lipschitz constant $d$. 
For each $n$, we denote the $x$-coordinate of $Q_{n}$ by $b^{\prime}_{n}$.  
Suppose $x_{1}, x_{2}$ are reals such that $x_{1} < x_{2} < \beta$. 
Let $n$ be the least $n$ such that $x_{2} \leq b^{\prime}_{n}$. In the case of $n=0$ it holds that $|g(x_{2}) - g(x_{1})|=0 \leq d |x_{2} - x_{1}| $. 
Otherwise, we have $n \geq 1$ and the following holds. 

\begin{equation} \label{eq:solo2l1upperbound1}
\left| \frac{ g(x_{2}) - g(x_{1}) }{x_{2} - x_{1}} \right|
\leq 
\max \{ \left| \frac{ g(b^{\prime}_{j+1} ) - g( b^{\prime}_{j} ) }{b^{\prime}_{j+1} - b^{\prime}_{j} } \right| : 
j+1 \leq n \}
\end{equation}

Recall that for each $j$, $Q_{j} (b^{\prime}_{j}, f( b^{\prime}_{j} ))$ is an interior point of $D_{b^{\prime}_{j+1}, f( b^{\prime}_{j+1} )}$ 
and the suffix $ (b^{\prime}_{j+1}, f( b^{\prime}_{j+1} ))$ of $D$ equals the coordinate of $Q_{j+1}$. 
Therefore, $b^{\prime}_{j} < b^{\prime}_{j+1} $, and 
$-d (b^{\prime}_{j+1} - b^{\prime}_{j} ) + g( b^{\prime}_{j+1}  ) 
< g( b^{\prime}_{j}  ) < g( b^{\prime}_{j+1}  )$. Therefore, for each $j$, the following holds. 

\begin{equation} \label{eq:solo2l1upperbound2}
\left| \frac{ g (b^{\prime}_{j+1} ) - g ( b^{\prime}_{j} ) }{ b^{\prime}_{j+1} - b^{\prime}_{j} } \right| 
< d
\end{equation}

Hence, the left-hand side of \eqref{eq:solo2l1upperbound1} is less than the right-hand side of \eqref{eq:solo2l1upperbound2}. 
Thus, $| g ( x_{2} ) - g ( x_{1} ) | \leq d | x_{2} - x_{1} |$. 
Therefore, $(L)_{1}$ holds. 
\end{proof}

\begin{coro} \label{coro:lip1}
Suppose that $\alpha$ and $\beta$ are left-c.e. reals. 
Then $\alpha \leq_{S} \beta$ if and only if there exists a rational $s < \alpha$ 
and a function $f: [s, \beta] \to \mathbb{R}$ of the following properties. 

{\rm (a)} $f$ is computable in the sense of Weihrauch. 

{\rm (b)} $f$ is Lipschitz continuous in $[s, \beta]$. 

{\rm (e)} There exists a strict increasing sequence of rationals $\{ r_{n} \}$ such that ($s \leq r_{n}$, and) $r_{n} \to \beta - 0$ and $f(r_{n}) \to \alpha - 0$. 
Here, $\{ r_{n} \}$ may be non-computable.
\end{coro}

\begin{coro} \label{coro:lip2}
Suppose that $\alpha$ and $\beta$ are left-c.e. reals. 
Then $\alpha \leq_{S} \beta$ if and only if there exists a function 
$f: (-\infty, \beta) \to (-\infty, \alpha)$ of the following properties. 

{\rm (a)} $f$ is computable in the sense of Weihrauch. 

{\rm (e)} There exists a strict increasing sequence of rationals $\{ r_{n} \}$ such that $r_{n} \to \beta - 0$ and $f(r_{n}) \to \alpha - 0$. 
Here, $\{ r_{n} \}$ may be non-computable.

{\rm (b${}^{\prime}$)} The condition for Lipschitz continuity holds whenever the larger point is some $r_{n}$. More precisely, there exists a positive real number $L$ such that for any $x  < \beta$ and any natural number $n$, 
if $x < r_{n}$ then $(f(r_{n}) - f(x)) \leq L (r_{n} - x)$. 
\end{coro}

\begin{proof}
In the proof of Theorem~\ref{theo:lip1}, the proof of ${(L)}_{2} \Rightarrow \alpha \leq_{S} \beta$ works in the present setting. 
\end{proof}

We can extend Corollary~\ref{coro:lip2} to the case of qS-reduction and H\"{o}lder continuity. 

\begin{lemm} \label{lemm:lip2}
Suppose that $\alpha$ and $\beta$ are left-c.e. reals. 
Then $\alpha \leq_{qS} \beta$ if and only if there exists a function 
$f: (-\infty, \beta) \to (-\infty, \alpha)$ of the following properties. 

{\rm (a)} $f$ is computable in the sense of Weihrauch. 

{\rm (e)} There exists a strict increasing sequence of rationals $\{ r_{n} \}$ 
 such that $r_{n} \to \beta - 0$ and $f(r_{n}) \to \alpha - 0$. 
Here, $\{ r_{n} \}$ may be non-computable.

{\rm (b${}^{\prime\prime}$)} The condition for H\"{o}lder continuity holds whenever the larger point is some $r_{n}$. More precisely, there exists a positive real number $H$ and a positive integer $\ell$ such that for any $x  < \beta$ and any natural number $n$, 
if $x < r_{n}$ then $(f(r_{n}) - f(x))^{\ell} \leq H (r_{n} - x)$. 
\end{lemm}

\begin{proof}
The above assertion implies $\alpha \leq_{qS} \beta$: 
The proof is very similar to the counterpart in the proof of Theorem~\ref{theo:lip1}. 

$\alpha \leq_{qS} \beta$ implies the above assertion: 
Suppose $\alpha \leq_{qS} \beta $. There exist a partial computable function $f$ and positive integers $d$ and $\ell$ with the following properties. 
For each rational $q < \beta$, $f(q) \downarrow < \alpha$ and we have 
$(\alpha - f(q))^{\ell} < d (\beta - q)$. 
We are going to modify the proof of ``$\alpha \leq_{S} \beta \Rightarrow (L)_{1}$'' (a part of Theorem~\ref{theo:lip1}). 
As before, let $\{ b_{n} \}_{n}$ be a computable sequence of rationals that increasingly converges to $\beta$, and assume that $f(b_{n}) < f(b_{n+1})$. 
We investigate the following new region $E_{w,z}$ in place of $D_{w,z}$. 
\begin{equation}
E_{w,z} =
\{ (x,y) \in \mathbb{R}^{2} : 
-d^{1/\ell} (w-x)^{1/\ell} +z \leq y \leq z \}
\end{equation}

Let $R_{i} (i \in \mathbb{N})$ be the new $Q_{i} (i \in \mathbb{N})$ defined by means of $E_{w,z}$. By connecting $R_{i}$s we get a line graph, and we define $h(x)$, the counterpart to $g(x)$, by this line graph. For each natural number $i$, let $(b^{\ast}_{i}, f( b^{\ast}_{i} ))$ be the coordinate of $R_{i}$. 
By our definition, $b^{\ast}_{0} = b_{0}$. 
In Figure~\ref{fig:graph_ps2hol}, the horizontal line leading to $(\beta, \alpha)$ denotes the line $y = \alpha, x \leq \beta$. The curve starting from $(\beta, \alpha)$ denotes the curve  $y=-d(\beta -x)^{1/\ell} + \alpha $.  The region between the horizontal line and the curve is $E_{\beta, \alpha}$. The thick half line of Figure~\ref{fig:graph_ps2hol} is the set $\{ (x, f(b_{0})) : x \leq b_{0 }\}$.  

\begin{figure}[htb] 
\centering
\includegraphics[width=.5\textwidth]{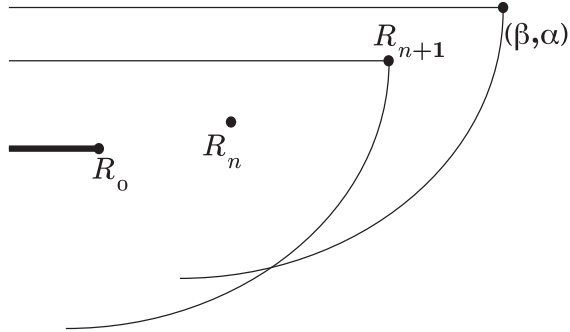}
\caption{Construction of the graph of $h$\label{fig:graph_ps2hol}}
\end{figure}

Suppose that $n$ is a natural number and $x$ is a real number such that $x < b_n^{\ast}$. 
In the case of $n=0$, it holds that $h(b_n^{\ast}) = h(x) = f(b_{0})$. 
Therefore the inequality $( h(b_n^{\ast}) - h(x) )^{\ell} \leq d (b_n^{\ast} - x)$ is apparent.

If $n > 0$ then all of the $R_{i}$s $(i = 0, \dots, n-1)$ are internal points of $E_{b^{\ast}_{n}, f( b^{\ast}_{n} )}$. There are two cases. Case 1: Point $(x, h(x))$ is on the half line ``$x \leq b^{\ast}_{0}$ and $y = f( b^{\ast}_{0} ) $''. Case 2: Point $(x, h(x))$ is on line segment $R_{i}R_{i+1}$ for some $i < n$. In the both cases, point $(x, h(x))$ is an internal point of  $E_{b^{\ast}_{n}, f( b^{\ast}_{n} )}$. Therefore the inequality $( h(r_{n}) - h(x) )^{\ell} \leq d (r_{n} - x)$ holds. 
\end{proof}


Now we state and prove the main theorem.

\begin{theo} \label{theo:hol1}
Suppose that $\alpha$ and $\beta$ are left-c.e. real numbers. 
Then $\alpha \leq_{qS} \beta$ holds if and only if there exists a function $g : [0,\beta) \to [0,\alpha)$ with the following properties. 
\begin{enumerate}
\item $g$ is nondecreasing. 

\item $g$ is computable in the sense of Weihrauch. 

\item If $q \to \beta$ then $g(q) \to \alpha$. 

\item $g$ is H{\"o}lder continuous with positive order $\leq 1$. 
\end{enumerate}
\end{theo}
\begin{proof}
The ``if'' direction ($\Leftarrow$) is immediate from Lemma \ref{lemm:lip2}. 
We are going to prove ``only if'' direction ($\Rightarrow$). 
Suppose that $\alpha \leq_{qS} \beta$. 
Suppose that a partial computable function $f$ and positive integers $d$ and $\ell$ satisfy the following property. For each rational number $q < \beta$ it holds that $f(q) \downarrow < \alpha$ and $(\alpha - f(q))^{\ell} < d (\beta - q)$. 
In the case where $\ell = 1$ the assertion reduces to Theorem~\ref{theo:lip1}. 
Throughout the rest of the proof, we assume $\ell \geq 2$. 
Assume that sequence $\{ r_{n} \}$ ($\nearrow \beta$) and 
Weihrauch computable function $h$ are those constructed in the proof of Lemma~\ref{lemm:lip2}. 
We construct a new function $g$ based on $h$. 

Let $s = 1 /\ell$. By our assumption of $\ell \geq 2$, we have $0 < s < 1$. 
For every natural number $n$, the following holds. 
\begin{equation} \label{eq:hol1thefunctionh}
\forall k < n ~ [h(r_n) - h(r_k) < d(r_n - r_k)^s]
\end{equation}

We are going to define a real number $t_n$ as a solution of the following equation in variable $x$. 
\begin{equation} \label{eq:hol1tnequation1}
h(r_n) + d(x - r_n)^s = h(r_{n+1}) + d(x - r_{n+1})^s
\end{equation}

The equation \eqref{eq:hol1tnequation1} has a solution $> r_{n+1}$ by the following reason. The equation \eqref{eq:hol1tnequation1} is equivalent to the following. 
\begin{equation} \label{eq:hol1tnequation2}
d(x - r_n)^s - d(x - r_{n+1})^s = h(r_{n+1}) - h(r_n)
\end{equation}
For a while, let $h_{L} (x)$ denote the left-hand side of \eqref{eq:hol1tnequation2}. 
On the one hand, $h_{L} (r_{n+1}) = d (r_{n+1} - r_{n})^{s}$ is, by \eqref{eq:hol1thefunctionh}, larger than $h(r_{n+1}) - h(r_n) > 0$. 
On the other hand, $h_{L}(x) \to 0+$ if $x \to \infty$, which is verified by means of L'Hospital's rule. 
Hence, by the intermediate value theorem, 
 \eqref{eq:hol1tnequation1} has a solution $> r_{n+1}$. 
Let $t_{n}$ be the solution. 

We define $A_{n}$ as to be $h(r_{n+1}) + d(t_n - r_{n+1})^s$. 
Thus $A_{n} =h(r_{n+1}) + d(t_n - r_n)^s$. 
By means of $t_n$ and $A_n$, we define a function $g_n$ on the closed interval $[r_n,r_{n+1}]$. 
\begin{equation} \label{eq:hol1tnequation3}
g_n(x) := A_n -d(t_n - x)^s.
\end{equation}

It is immediate that $g_{n} (r_{n+1}) = g_{n+1} (r_{n+1})$. 
We define a continuous function $g$ by connecting the graphs of $g_n$($n \in \mathbb{N}$). 
It is not hard to see that $t_{n}$ is a computable real number. Therefore we know that $g$ is a computable function in the sense of Weihrauch. Thus the assertion (2) of the theorem holds. 
Of course, $g$ is nondecreasing, and $g(x)$ approaches to $\alpha = \lim_{j \to \infty} f(r_{j})$ when $x \to \beta$. Thus the assertions (1) and (3) of the theorem hold. 

Now we are going to show that 
$g$ is H{\"o}lder continuous with positive order $\leq 1$. 
Given a positive real number $\varepsilon$ and for each positive real numbers $x,y$ such that $ y=x+\varepsilon < 1$, 
we are going to show that $g(y) - g(x) < 3 d \varepsilon^s$. 

Case 1: $x$ and $y$ are in the same interval. To be more precise, $x,y \in [r_n,r_{n+1})$ for some $n \in \mathbb{N}$. 
\begin{align} \label{eq:hol1tnequation4}
g(y) - g(x) &= g_{n} (x+\varepsilon) - g_{n}(x) \notag \\ 
&= (A_{n} - d (t_{n} - x - \varepsilon)^{s}) - (A_{n} - d (t_{n} - x)^{s}) \notag \\
&= d( (t_{n} - x)^s - (t_{n} - x - \varepsilon)^s ) 
\end{align}
By means of the inequality $0 < s < 1$, 
it is not hard to see that for any real numbers $z$ and $w$, 
if $0 < z < w$ then $w^s - z ^s \leq (w - z)^s$. 
In particular, the last formula of \eqref{eq:hol1tnequation4} is at most 
$d\varepsilon^s$. 
In summary, it holds that $g(y) - g(x) \leq d \varepsilon^{s}$.

Case 2: otherwise. Then for some $k$ and $n$ such that $k < n$, 
it holds that $r_k \leq x < r_{k+1} < r_{n} \leq y < r_{n+1}$. 
Let $a := r_{k+1} - x$, $b := r_{n+1} - r_{n+1}$ and $c := y - r_{n+1}$. 

The inequalities \eqref{eq:hol1tnequation5} and \eqref{eq:hol1tnequation7} reduce to Case 1. 
Recall that $g$ and $h$ have the same value at an end point of each interval, that is, $g(r_{j}) = h(r_{j})$ for each natural number $j$. 
Hence the inequality \eqref{eq:hol1tnequation6} reduces to \eqref{eq:hol1thefunctionh}. 

\begin{eqnarray}
g(r_{k+1}) - g(x) < da^s \label{eq:hol1tnequation5}
\\
g(r_n) - g(r_{k+1}) < db^s \label{eq:hol1tnequation6}
\\
g(y) - g(r_n) < dc^s \label{eq:hol1tnequation7}
\end{eqnarray}

Therefore, we have the following. 
\begin{equation} \label{eq:hol1tnequation8}
g(y) - g(x) < d(a^s + b^s + c^s)
\end{equation}

In order to complete Case 2, we are going to employ H{\"o}lder's inequality \cite{BB1961}. 
Under the assumption of $p,q >1$ and $1/p + 1/q = 1$, 
for any nonnegative real numbers $a_1,a_2,a_3,b_1,b_2$ and $b_3$, 
the following holds. 
\begin{equation} \label{eq:holderineq}
a_1b_1 + a_2b_2 + a_3b_3 \leq (a_1^p + a_2^p + a_3^p)^{1/p}(b_1^q + b_2^q + b_3^q)^{1/q}
\end{equation}

For our purpose, we investigate the case where $1/p := s$, $1/q := 1 - s$, $a_1 := a^s$, $a_2 := b^s$, $a_3 := c^s$, and $b_1 = b_2 =b_3 := 1$. 
In this particular setting, the inequality \eqref{eq:holderineq} is the following. 
\begin{equation}
a^s + b^s + c^s \leq 3^{1-s}(a + b + c)^s.
\end{equation}
Hence, we have the following. 
\begin{align}
g(y) - g(x) &< 3^{1-s}d(a + b + c)^s\\ \notag
&= 3^{1-s}d\varepsilon^s
\end{align}
Thus in the both Cases, we have $g(y) - g(x) < 3 d \varepsilon^s$. 
This completes the proof of the assertion (4).
\end{proof}


\bigskip

\begin{thebibliography}{99}
\bibitem{BB1961} 
Beckenbach, E. F.  and Bellman, R.
\emph{Inequalities}. Springer, Berlin, 1961.

\bibitem{CCHK1999} 
Calude, C., Coles, R. J., Hertling, P. H. and Khoussainov, B. 
Degree-theoretic aspects of computably enumerable reals. 
In: Cooper, S. B. and Truss, J. K. eds. Models and Computability. 
London Mathematical Society Lecture 259, 
Cambridge University Press, Cambridge, 1999.

\bibitem{DH2010} Downey, R.G. and Hirschfeldt, D.R. 
\emph{Algorithmic Randomness and Complexity}. 
Springer, New York, 2010.

\bibitem{G1957} Grzegorczyk, A. 
On the definitions of computable real continuous functions. 
\emph{Fund. Math.} 1957, 61-71. 

\bibitem{HHSY2013}Higuchi, K.,Hudelson, W. M. P.,Simpson, S. G.,Yokoyama, K.
Propagation of partial randomness.
\emph{Annals of Pure and Applied Logic} 2013, 11-12.

\bibitem{KTZ2018} Kawamura, A., Thies, H. and Ziegler, M. 
Average-case polynomial-time computability of Hamiltonian dynamics. 
In: \emph{43rd International Symposium on Mathematical Foundations of Computer Science (MFCS 2018)}, Leibniz International Proceedings in Informatics, 2018, 30:1-30:17. 
DOI: 10.4230/LIPIcs.MFCS.2018.30

\bibitem{KF1982} Ko, Ker-I., Friedman, H. 
Computational complexity of real functions. 
\emph{Theoret. Comput. Sci.} 1982, 323-352. 

\bibitem{MNS2018} Miyabe, K., Nies, A. and Stephan, F.
Randomness and Solovay degrees.
\emph{Journal of Logic and Analysis 10:3} 2018, 1-13.

\bibitem{N2009}Nies, A. \emph{Computability and Randomness}.
Oxford University Press, Oxford, 2009.

\bibitem{So1975} Solovay, R.M.
Draft of paper (or series of papers) on Chaitin's work. 
Unpublished notes, May 1975, 215 pages.

\bibitem{St2005} Staiger, L. 
Constructive dimension equals Kolmogorov complexity. 
\emph{Inform. Comput. Lett.} 2005, 149-153.

\bibitem{T1999} Tadaki, K.
Algorithmic information theory and fractal sets. 
in: \emph{Proceedings of 1999 Workshop on Information-Based Induction Sciences (IBIS’99)}, 1999, 105-110.

\bibitem{T2002} Tadaki, K. 
A generalization of Chaitin's halting probability and halting self-similar sets.
\emph{Hokkaido Math. J., vol. 31} 2002, 219-253.

\bibitem{T2009} Tadaki, K. 
Partial randomness and dimension of recursively enumerable reals.
In: \emph{Proceedings of the 34st International Symposium on Mathematical Foundations of Computer Science (MFCS 2009), Lecture Notes in Computer Science}, vol. 5734, Springer, 2009, 687-699.

\bibitem{We2000} Weihrauch, K.
\emph{Computable analysis: an introduction}. 
Springer, New York, 2000.
\end{thebibliography}
\end{document}